\documentclass{amsart}

\usepackage[T1]{fontenc}
\usepackage[latin9]{inputenc}
\usepackage{multirow}
\usepackage{array}
\usepackage{bm}
\usepackage{amsmath, amssymb}
\usepackage{mathtools}
\usepackage{pb-diagram}
\usepackage{graphicx}

\usepackage{amsthm} 

 \usepackage{bm}

\theoremstyle{plain} 
\newtheorem{theorem}{Theorem}[section]
\newtheorem{lemma}[theorem]{Lemma}
\newtheorem{corollary}[theorem]{Corollary}
\newtheorem{proposition}[theorem]{Proposition}
\newtheorem{fact}[theorem]{Fact}
\newtheorem{theor}{Theorem}

\theoremstyle{definition}
\newtheorem{definition}[theorem]{Definition}
\newtheorem{remark}[theorem]{Remark}
\newtheorem{example}[theorem]{Example}

\renewcommand{\Re}{\operatorname{Re}}
\renewcommand{\Im}{\operatorname{Im}}
\newcommand{\phoro}[1]{\mathtt{#1}}

\newcommand{\psin}{\operatorname{\phoro{sin}}}
\newcommand{\pcos}{\operatorname{\phoro{cos}}}
\newcommand{\ptan}{\operatorname{\phoro{tan}}}

 \usepackage{caption}

\usepackage{color}
\usepackage{graphicx,color}
\usepackage{amsmath, amssymb, graphics}

\begin{document}
\title[Isometric and anti-isometric classes of timelike minimal surfaces]{Isometric and anti-isometric classes of timelike minimal surfaces in Lorentz--Minkowski space}
  \author[S. Akamine]{Shintaro Akamine}
\address[Shintaro Akamine]{
Department of Liberal Arts, College of Bioresource Sciences,
Nihon University, 
1866 Kameino, Fujisawa, Kanagawa, 252-0880, Japan}
\email{akamine.shintaro@nihon-u.ac.jp}


\thanks{
This work was partially supported by JSPS KAKENHI Grant Numbers 19K14527 and 23K12979. 
}

\begin{abstract}
Isometric class of minimal surfaces in the Euclidean $3$-space $\mathbb{R}^3$ has the rigidity: if two simply connected minimal surfaces are isometric, then one of them is congruent to a surface in the specific one-parameter family, called the associated family, of the other. 
On the other hand, the situation for surfaces with Lorentzian metrics is different. In this paper, we show that there exist two timelike minimal surfaces in the Lorentz-Minkowski $3$-space $\mathbb{R}^3_1$ that are isometric each other but one of which does not belong to the congruent class of the associated family of the other. We also prove a rigidity theorem for isometric and anti-isometric classes of timelike minimal surfaces under the assumption that  surfaces have no flat points. 

Moreover, we show how symmetries of such surfaces propagate for various deformations including isometric and anti-isometric deformations. In particular, some conservation laws of symmetry for Goursat transformations are discussed.
\end{abstract}

\maketitle

\section{Introduction}
Surfaces which admit a one-parameter family of isometric deformations preserving the mean curvature are called {\it Bonnet surfaces}. Due to Bonnet \cite{Bonnet}, it is known that any constant mean curvature surfaces in the Euclidean space $\mathbb{R}^3$ which is not totally umbilic is a Bonnet surface. 
For the case of minimal surfaces in $\mathbb{R}^3$, each minimal surface has a one-parameter family of isometric minimal surfaces, called the {\it associated family}.
Furthermore, the following rigidity theorem was shown by Schwarz \cite[p.175]{Schwarz}:

\begin{fact}\label{thm:Schwarz}
If two simply connected minimal surfaces in $\mathbb{R}^3$ are isometric, then one of them is congruent to a surface in the associated family of the other.
\end{fact}

However, the situation in the case of surfaces with Lorentzian metrics is different as follows.
\begin{itemize}
\item Not only isometric deformations but also anti-isometric deformations, which reverse the first fundamental form of the original surface can be considered.
\item The shape operator is not necessarily diagonalizable and hence a specific point,  
the so-called a {\it quasi-umbilic point}, on which the shape operator is non-diagonalizable can appear on such a surface (see Section 2.2). Moreover umbilic and quasi-umbilic points are not isolated in general even for the case of minimal surfaces. 
\end{itemize}

In this paper, we consider {\it timelike minimal surfaces} in the Lorentz-Minkowski 3-space $\mathbb{R}^3_1$ with signature $(-,+,+)$, which are surfaces whose induced metric from $\mathbb{R}^3_1$ is Lorentzian and whose mean curvature vanishes identically. The following is the first main theorem.

\begin{theor}\label{Theorem1_Intro}
The following statement holds.
\begin{itemize}
\item[(1)] There exist two flat timelike minimal surfaces in $\mathbb{R}^3_1$ which are isometric each other but one of which does not belong to the congruent class of the associated family of the other.
\end{itemize}

Let $f_1$ and $f_2$ be simply connected timelike minimal surfaces in $\mathbb{R}^3_1$ without flat points. Furthermore, the following statements hold.
\begin{itemize}
\item[(2)] If $f_1$ and $f_2$ are isometric, then $f_1$ is congruent to a surface in the associated family $\{(f_2)_\theta\}_{\theta \in \mathbb{R}}$ of $f_2$.
\item[(3)] If $f_1$ and $f_2$ are anti-isometric, then $f_1$ is congruent to a surface in the associated family $\{(\hat{f_2})_\theta\}_{\theta \in \mathbb{R}}$ of the conjugate surface $\hat{f_2}$.
\end{itemize}
\end{theor}
The statement (1) gives a counterexample to show that the same assertion as in Fact \ref{thm:Schwarz} for timelike minimal surfaces does not hold. The statements (2) and (3) show rigidities for isometric and anti-isometric classes of timelike minimal surfaces under the assumption that surfaces have no flat points, where a flat point is a point on which the Gaussian curvature $K$ vanishes. We remark that flat points of a timelike minimal surface consist of umbilic and quasi-umbilic points. The definitions of the associated family and the conjugate surface of a timelike minimal surface will be given in Definition \ref{def:associated}.

In the second half of the paper, we consider symmetry of timelike minimal surfaces. We show how symmetries of such surfaces propagate under various deformations including the above isometric and anti-isometric deformations. 
Each conformal timelike minimal surface $f\colon M \to \mathbb{R}^3_1$ from a (simply connected) Lorentz surface $M$ into $\mathbb{R}^3_1$ is realized as the real part of a paraholomorphic null curve $\Phi\colon M \to \mathbb{C}'^3$ into the paracomplex $3$-space $\mathbb{C}'^3$, that is, $f=\Re{\Phi}$. More specifically, a Weierstrass type representation formula was given by Konderak \cite{Konderak}, see Fact \ref{fact:Weierstrass1} for more details. 
Since the conformality is preserved under transformations of the form $f_A:=\Re{A\Phi}$, which is called the {\it Goursat transformation} of $f$ (see \cite{G}), for a matrix $A$ in the paracomplex conformal group
\[
\mathrm{CO}(1,2; \mathbb{C}') = \{ A\in \mathrm{M}(3,\mathbb{C}') \mid  {}^t\! AI_{1,2}A=cI_{1,2},\ c\in \mathbb{C}',\ c\bar{c}\neq 0 \},
\]
where ${}^t\! A$ is the transposed matrix of $A$ and $I_{1,2}=\text{diag}(-1,1,1)$.
A symmetry $g$ of $f$ is an isometry of the Lorentz surface $M$ satisfying $f\circ g = Of +t $ for some $O$ in the indefinite orthogonal group $O(1,2)$ of $\mathbb{R}^3_1$ and a vector $t\in \mathbb{R}^3_1$, and we call $O$ the {\it linear part} of $g$.  The set of such symmetries is denoted by $S_f(M)$ and is often referred to as the {\it space group} (see Definition \ref{def:space group}). In the above setting, we give the following conservation law of symmetry for Goursat transformations.

 \begin{theor}\label{thm:symmetry_Goursat_Introduction}
 Let $f\colon M\to \mathbb{R}^3_1$ be a simply connected timelike minimal surface, $f_A$ be its Goursat transformation for $A\in \mathrm{CO}(1,2; \mathbb{C}')$ and $g\in S_f(M)$ with the linear part $O$. Then the following statements hold.
 \begin{itemize}
 \item[(1)] When $g$ is orientation preserving,  there exists $\widetilde{O}\in \mathrm{O}(1,2)$ such that  $AO=\widetilde{O}A$ if and only if $g\in S_{f_A}(M)$ and its linear part is $\widetilde{O}$.
 
 \item[(2)] When $g$ is orientation reversing,  there exists $\widetilde{O}\in \mathrm{O}(1,2)$ such that  $AO=\widetilde{O}\bar{A}$ if and only if $g\in S_{f_A}(M)$ and its linear part is $\widetilde{O}$.

  \end{itemize}
 \end{theor}

Symmetry conservation under deformations of minimal surfaces has often been discussed. A well-known classical symmetry correspondence is the fact that the line symmetry with respect to a straight line on a minimal surface in $\mathbb{R}^3$ (or a timelike minimal surface in $\mathbb{R}^3_1$) corresponds to the planar symmetry of the conjugate surface with respect to a plane orthogonal to the line, see  \cite{DHS, Karcher, KKSY}, for example.  
Similarly, as discussed by Kim, Koh, Shin and Yang \cite{KKSY}, there is also a symmetry correspondence between shrinking singularity (also called conelike singularity) and folding singularity on timelike minimal surfaces in $\mathbb{R}^3_1$. 
All of these symmetries are due to the reflection principle, which are derived from orientation reversing isometries of the form $g(z)=\bar{z}$. Hence, these symmetry relations are obtained by considering the Goursat transformation $f_J$ for $J=jI_3$ in Theorem \ref{thm:symmetry_Goursat_Introduction}, where $j$ is the imaginary unit of $\mathbb{C}'$ and $I_3$ is the identity matrix. Furthermore, by considering the Goursat transformation $f_D$ by a special matrix $D$ in Section \ref{sec:duality}, these symmetries about lines, planes, shrinking singularities, and folding singularities can be unified in Corollary \ref{cor:quadruple}. See Example \ref{example:Enneper} for a concrete example and Figure \ref{Fig:Intro}.

For translation symmetry, Meeks \cite{Meeks} showed a necessary and sufficient condition for the conjugate minimal surface in the Euclidean space to have a translation symmetry. Leschke and Moriya \cite{LM} also revealed results on the conservation of translation symmetry of simple factor dressing of minimal surfaces in $\mathbb{R}^3$ and $\mathbb{R}^4$, which is also a special kind of Goursat transformations. 
Since translation symmetry corresponds to the case where $O=I_3$ in Theorem \ref{thm:symmetry_Goursat_Introduction}, we also obtain a conservation law of translation symmetry in Corollary \ref{cor:translation}.
 
Similarly, as an application of Theorem \ref{thm:symmetry_Goursat_Introduction}, by considering specific matrices in $\mathrm{CO}(1,2; \mathbb{C}')$ as
\[
J(\theta)=\begin{pmatrix}
\phoro{e}^{j\theta} & 0& 0 \\
0 & \phoro{e}^{j\theta} & 0 \\
0 & 0 & \phoro{e}^{j\theta}
\end{pmatrix},\quad 
\hat{J}(\theta)=\begin{pmatrix}
j\phoro{e}^{j\theta} & 0& 0 \\
0 & j\phoro{e}^{j\theta} & 0 \\
0 & 0 & j\phoro{e}^{j\theta}
\end{pmatrix}\ \text{ and }\ 
A(\lambda)=\begin{pmatrix}
\frac{\lambda+\frac{1}{\lambda}}{2} & j\frac{\lambda-\frac{1}{\lambda}}{2} & 0 \\
j\frac{\lambda-\frac{1}{\lambda}}{2} & \frac{\lambda+\frac{1}{\lambda}}{2} & 0 \\
0 & 0 & 1
\end{pmatrix}
\]
where $\theta\in \mathbb{R}$ and $\lambda >0$, we obtain symmetry relations for the isometric deformation $\{f_\theta\}=\{f_{J(\theta)}\}$, the anti-isometric deformation $\{\hat{f}_\theta\}=\{f_{\hat{J}(\theta)}\}$ in Corollary \ref{cor:Spacegroup}, which is the Lorentzian counter part of the result for minimal surfaces by Meeks \cite[Theorem 5.5]{Meeks} and for the deformations $\{f_{A(\lambda)}\}$ called the {\it L\'opez-Ros} deformations discussed in Section 4.3. 

\begin{figure}
\begin{center}
\includegraphics[clip,scale=0.5,bb=0 0 900 600]{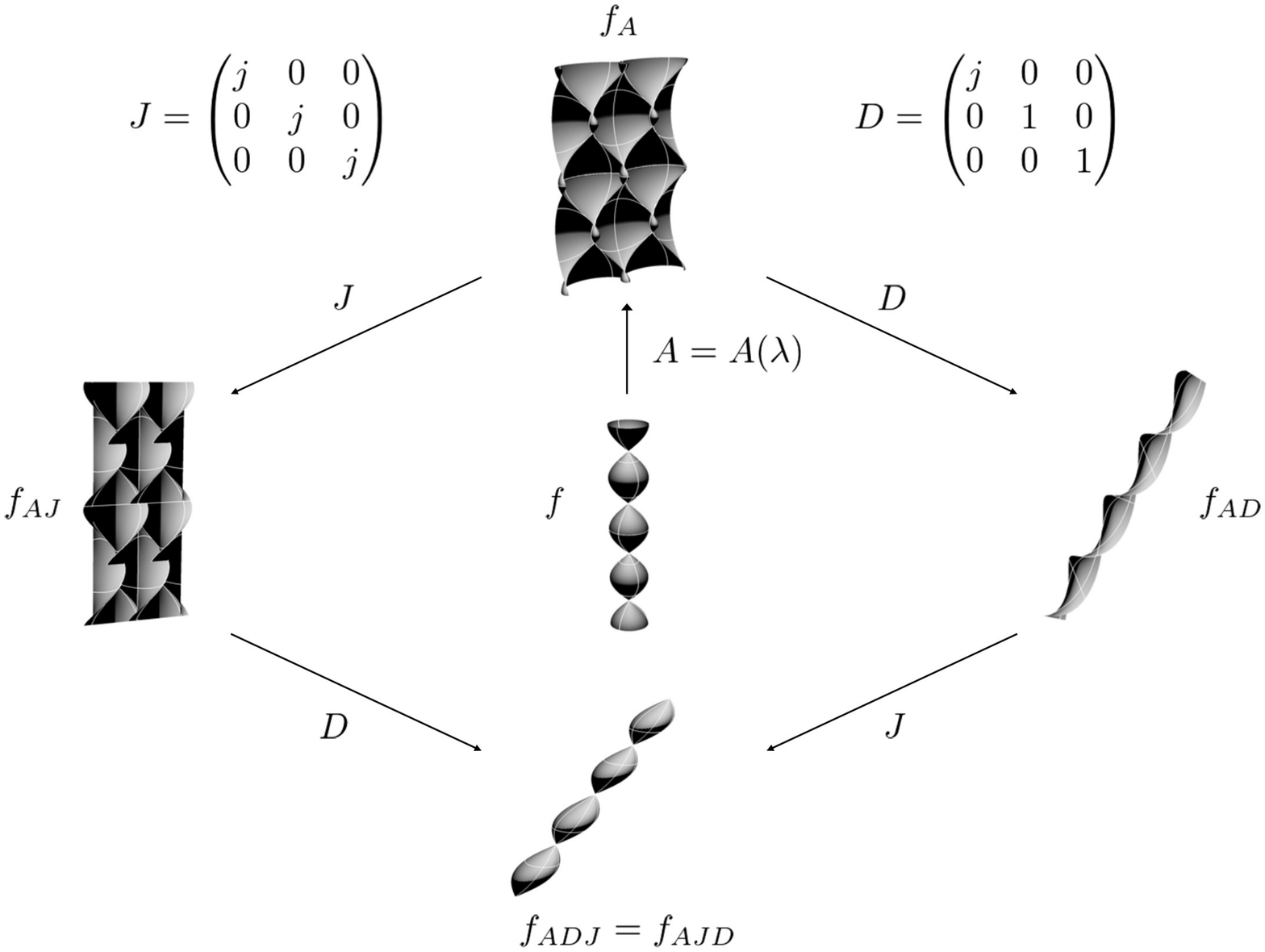}
\end{center}
\caption{Symmetries under Goursat transformations are related each others: two orientation preserving translation symmetries degenerating in one direction on the initial surface $f$ (the elliptic catenoid) are preserved as discussed in Corollary \ref{cor:translation}. In addition, four orientation reversing symmetries, a planar symmetry of $f_A$, a line symmetry of $f_{AJ}$, a folded symmetry of $f_{AD}$ and a point symmetry $f_{ADJ}$ also correspond to each other as discussed in Corollary \ref{cor:quadruple}.}\label{Fig:Intro}
\end{figure}

This article is organized as follows: In Section 2, we describe some notions of paracomplex analysis and timelike minimal surfaces. In Section 3, we investigate isometric and anti-isometric classes of timelike minimal surfaces. The proof of Theorem \ref{Theorem1_Intro} is given in two separate parts: the proofs of Proposition \ref{prop:counter} and Theorem \ref{thm:timelikeSchwarz}. Regarding (1) of Theorem \ref{Theorem1_Intro}, we also determine all flat timelike minimal surfaces in Proposition \ref{FlatMinimal_rev}. In Section 4, we give a proof of Theorem \ref{thm:symmetry_Goursat_Introduction} and several applications of it for well-known import deformations and transformations. Finally, in Section \ref{sec:Ex} we describe relationships between the symmetries of various concrete examples in terms of Goursat transformations. 

\section{Preliminary}
First, we briefly recall the theories of paracomplex analysis and timelike minimal surfaces.
For a more detailed introduction, we refer the readers to works such as \cite{A,ACO,IT,Konderak,W} and their references.

\subsection{Paracomplex analysis}

A {\it paracomplex number} (or {\it split-complex number}) is a number $z$ of the form $z=x+jy$, where $x,y\in \mathbb{R}$ and $j$ is the imaginary unit satisfying $j^2=1$. We denote the set of such paracomplex numbers by $\mathbb{C}'$ and refer it as the {\it paracomplex plane}.  Just as for complex numbers, $\mathbb{C}'$ forms an algebra over $\mathbb{R}$ and one can define the notions of 
\begin{itemize}
\item the \emph{real part} $\Re z := x$ and the \emph{imaginary part} $\Im z := y$ of $z=x+jy$, 
\item the {\it conjugate} $\bar{z}$ of $z=x+jy$ as $\bar{z}:=x-jy$, and 
\item the squared modulus of $z$ defined by $|z|^2 := z\bar{z}=x^2-y^2$.
\end{itemize}
It should be remarked that it is possible for the relation $|z|^2<0$ to hold, and the relation $|jz|^2=-|z|^2$ holds.

Given a paracomplex function $\phoro{f} : \Sigma \subset \mathbb{C}' \to \mathbb{C}'$ where $\Sigma$ is a simply-connected domain, we call $\phoro{f}$ \emph{paraholomorphic} if $\phoro{f}$ satisfies
	\begin{equation}\label{eqn:cauchyRiemann}
		\phoro{f}_{\bar{z}} = \partial_{\bar{z}} \phoro{f} = 0,
	\end{equation}
where $\partial_z := \frac{1}{2}\left(\partial_x + j \partial_y \right)$ and $\partial_{\bar{z}} := \frac{1}{2}\left(\partial_x - j \partial_y \right)$ are the paracomplex Wirtinger derivatives. 
Furthermore, we call a function $\phoro{f} : \Sigma \to \mathbb{C}'$ \emph{parameromorphic} if it is $\mathbb{C}'$-valued on an open dense subset of $\Sigma$ and for arbitrary $p \in \Sigma$, there exists a paraholomorphic function $\phoro{g}$ such that $\phoro{fg}$ is paraholomorphic near $p$.

We recall some elementary paracomplex analytic functions used in this paper. The exponential function $\phoro{e}^z$ is defined by
	\[
		\phoro{e}^z := \sum_{n = 0}^{\infty} \frac{z^n}{n!}.
	\]
We can see the paracomplex version of Euler's formula
	\[
		\phoro{e}^{j \theta} = \cosh{\theta} + j \,\sinh{\theta},\quad \theta \in \mathbb{R}.
	\]
By using this function, each element of the hyperbola $\{z\in \mathbb{C}' \mid |z|^2 =1\}$ or $\{z\in \mathbb{C}' \mid |z|^2 =-1\}$ is written as $z=\pm \phoro{e}^{j \theta}$ or $z=\pm j\phoro{e}^{j \theta}$ for some $\theta$, respectively, which will play an important role of the proof of Theorem \ref{thm:timelikeSchwarz}. 

The paracomplex circular functions are also defined via analytic continuation from the real counterparts as follows.
\[
\pcos z := \sum_{n = 0}^{\infty} (-1)^n \frac{z^{2n}}{(2n)!},
				\quad
			\psin z := \sum_{n = 0}^{\infty} (-1)^n \frac{z^{2n + 1}}{(2n + 1)!},\quad
			\ptan z := \frac{\psin z}{\pcos z}.
\]
The functions $\pcos z$ and $\psin z$ are paraholomorphic on $\mathbb{C}'$ satisfying the relations
\[
\pcos{z}= \cos{x}\cos{y}-j\sin{x}\sin{y}, \quad \psin{z}=\sin{x}\cos{y}+j\cos{x}\sin{y},
\]
where $z=x+jy$. The function $\ptan{z}$ is parameromorphic on $\mathbb{C}'$.

\subsection{Timelike minimal surfaces and their symmetry groups}
Let $\mathbb{R}^3_1$ be the Lorentz-Minkowski 3-space with the indefinite inner product
\[
\langle \prescript{t\!}{}(x_1,x_2,x_3), \prescript{t\!}{}(y_1,y_2, y_3)\rangle = -x_1y_1+x_2y_2+x_3y_3.
\]
A surface $f\colon M \to \mathbb{R}^3_1$ from a $2$-dimensional manifold $M$ to $\mathbb{R}^3_1$ is said to be {\it timelike} if its first fundamental form $\mathrm{I}_f=\langle df, df\rangle$ is Lorentzian. The Gaussian curvature $K$ and the mean curvature $H$ of a timelike surface $f$ are defined as follows.
\[
K := \det{S},\quad H:=\frac{1}{2}\mathrm{tr}{S},
\]
where $S$ is the shape operator of $f$. One of the remarkable properties of timelike surfaces is the diagonalizability of the shape operator. The shape operator $S$ of a timelike surface is not always diagonalizable over $\mathbb{R}$, that is, principal curvatures can be complex numbers. More precisely, there are three possibilities of the diagonalizability of $S$ at each point of a timelike surface in $\mathbb{R}^3_1$ as follows:
\vspace{0.2cm}

 \begin{itemize}
\item[(i)] $S$ is diagonalizable over $\mathbb{R}$. In this case $H^2 - K \geq 0$ with the equality holds on umbilic points.
\item[(ii)] $S$ is diagonalizable over $\mathbb{C}\setminus \mathbb{R}$. In this case $H^2 - K < 0$. 
\item[(iii)] $S$ is non-diagonalizable over $\mathbb{C}$. In this case $H^2 - K = 0$. Each point satisfying this condition is called {\it quasi-umbilic} (see \cite{Clelland}).
\end{itemize}
\vspace{0.2cm}

In this paper, we discuss {\it timelime minimal surfaces}, which are timelike surfaces with $H=0$. Hence, the above relations show that the diagonalizability of $S$ is directly related to the sign of $K$, and there is no restriction of the sign of $K$ for timelike minimal surfaces. This is quite different from the situation where $K\leq 0$ holds for minimal surfaces in $\mathbb{R}^3$ and $K\geq 0$ holds for maximal surfaces in $\mathbb{R}^3_1$.

By using paracomplex analysis, timelike minimal surfaces admit a Weierstrass type representation \cite{Konderak} on Lorentz surfaces (see also \cite{A} for surfaces with singularities):

\begin{fact}\label{fact:Weierstrass1}
	Any timelike minimal surface $f\colon M \to \mathbb{R}^3_1$ can be represented as
		\begin{equation}\label{eq:pW}
			f(z) = \Re\int \prescript{t\!}{}(-(1 + h^2),j(1 - h^2), 2h)\eta
		\end{equation}
	over a simply-connected Lorentz surface $M$ on which $h$ is a parameromorphic, while $\eta$ and $h^2\eta$ are paraholomorphic.
	Furthermore, the induced metric of the surface becomes
		\begin{equation}\label{eqn:wConformal}
			 \mathrm{I}_{f} = -(1 - |h|^2)^2 |\eta |^2.
		\end{equation}
	We call $(h, \eta)$ the \emph{Weierstrass data} of the timelike minimal surface $f$.
\end{fact}

\begin{remark}\label{remark:Gaussmap}
The parameromorphic function $h$ is identified with the Gauss map of $f$ as follows.
Let $S^2_1=\{\prescript{t\!}{}(x_1,x_2,x_3)\in \mathbb{R}^3_1 \mid -x_1^2+x_2^2+x_3^2=1\}$ be the unit pseudosphere and $H^1=\{\prescript{t\!}{}(x_1,x_2)\in \mathbb{R}^2_1 \mid -x_1^2+x_2^2=-1\}$ be the hyperbola on the $x_1x_2$-plane, identified with the Minkowski plane $\mathbb{R}^2_1$. We consider the following stereographic projection $\mathcal{P}$ with respect to the point $(0,0,1)$
\[
\mathcal{P}(\bm{x}) = \prescript{t\!}{}{\left(\frac{x_1}{1-x_3}, \frac{x_2}{1-x_3}\right)},\quad \bm{x}=\prescript{t\!}{}(x_1,x_2,x_3)
\]
from $S^2_1\setminus \{x_3\neq 1\}$ to $\mathbb{R}^2_1\setminus H^1$. 
Since we can take a unit normal vector field $\nu$ of \eqref{eq:pW} as
\[
\nu=\frac{1}{1-|h|^2}\prescript{t\!}{}{\left(2\Re h,2\Im h, 1+|h|^2\right)},
\]
we have the relation $\mathcal{P}\circ \nu = h$ where we identify the $x_1x_2$-plane with $\mathbb{C}'$.

\end{remark}
Similar to the minimal surfaces and maximal surfaces cases, timelike minimal surfaces also admit associated families and conjugate surfaces as follows.
\begin{definition}[Associated family and conjugate surface]\label{def:associated}
Let $f$ be a timelike minimal surface written as \eqref{eq:pW} with Weierstrass data $(h,\eta)$, we define the {\it associated family} $\{f_\theta\}_{\theta \in \mathbb{R}}$ consists of the timelike minimal surface $f_\theta$ which is defined by the Weierstrass data $(h,\phoro{e}^{j\theta}\eta)$. 
We also call the timelike minimal surface $\hat{f}$ defined by the Weierstrass data $(h,j\eta)$ the {\it conjugate surface} of $f$.
\end{definition}

\begin{remark}\label{rema:asso_family}
Different from minimal and maximal surfaces, the conjugate surface $\hat{f}$ is not a member of the associated family $\{f_\theta\}_{\theta \in \mathbb{R}}$. 
By \eqref{eqn:wConformal} and the relation $|jz|^2=-|z|^2$, each $f_\theta$ is isometric to the original one $f=f_1$ and $\hat{f}$ is anti-isometric to $f$.
\end{remark}

The Weierstrass-type representation formula \eqref{eq:pW} gives a conformal parametrization for timelike minimal surfaces. In addition, timelike surfaces have the following characterization on {\it null coordinates} $(u,v)$, on which the first fundamental form of the surface is written as $ \mathrm{I}_{f} =\Lambda dudv$ for some function $\Lambda$.

\begin{fact}[\cite{McNertney}]\label{Fact:McNertney}
If $\varphi(u)$ and $\psi(v)$ are null curves in $\mathbb{R}^3_1$ such that $\varphi'(u)$ and $\psi'(v)$ are linearly independent for all $u$ and $v$, then
\begin{equation}\label{null curves decomposition}
f(u,v)=\varphi(u)+\psi(v)
 \end{equation}
is a timelike minimal surface. Conversely, any timelike minimal surface can be written locally as the equation \eqref{null curves decomposition} for some two null curves.
\end{fact}

\begin{remark}\label{rema:asso_family2}
We can easily check that the associated family $\{f_\theta\}_{\theta \in \mathbb{R}}$ and the conjugate surface $\hat{f}$ in Definition \ref{def:associated} correspond to deformations of the generating null curves $\varphi$ and $\psi$ in \eqref{null curves decomposition} as follows. 
\begin{align*}\label{null curves decomposition2}
f_\theta(u,v)=e^\theta \varphi(u)+e^{-\theta}\psi(v),\quad \hat{f}(u,v)= \varphi(u)- \psi(v).
 \end{align*}
These surfaces have the relationship $f_\theta =\cosh{\theta}f+\sinh{\theta}\hat{f}$.
 \end{remark}

In Section \ref{sec:Goursat}, we will see that symmetries of timelike minimal surfaces can be  controlled under various deformations including the above isometric and anti-isometric deformations. 
For this purpose, we define the group of symmetries of a timelike minimal surface based on the work \cite{Meeks}.

\begin{definition}[Space group]\label{def:space group}
Let $f\colon M \to \mathbb{R}^3_1$ be a timelike minimal surface. The {\it space group} $S_f(M)$ of $f$ is the group of isometries of $M$ induced by symmetries of $f(M)$ in $\mathbb{R}^3_1$, which consists of an isometry $g\colon M \to M$ such that $f(g(p))=Of(p)+t$ for a matrix $O\in O(1,2)$, a vector $t\in \mathbb{R}^3_1$  and arbitrary $p\in \mathbb{R}^3_1$. 
\[
  \begin{diagram}
    \node[2]{M}
    \arrow[2]{e,b}{f}
    \arrow[1]{s,l}{g}
    \node[2]{f(M)\subset \mathbb{R}^3_1}
    \arrow[1]{s,r}{\text{isometry of $\mathbb{R}^3_1$}} \\
    \node[2]{M} 
     \arrow[2]{e,t}{f}
    \node[2]{f(M)\subset \mathbb{R}^3_1}
  \end{diagram}
\]
We call the matrix part $O$ the {\it linear part} of $g\in S_f(M)$.
We also denote the orientation preserving subgroup of $S_f(M)$ by $S^\circ_f(M)$, and the orientation reversing elements of $S_f(M)$ by $S^r_f(M)$.
\end{definition}

\section{Rigidity theorem}

As we saw in Remark \ref{rema:asso_family} and Introduction, the associated family gives an isometric deformation and the converse is true as in the sense of Fact \ref{thm:Schwarz}. For the proof of Fact \ref{thm:Schwarz}, the fact that flat points of non-planar minimal surfaces are isolated and real analyticity play an important role, as one can see in \cite[p.275]{Spivak}, for example. By the same reasons, the same result obviously holds for maximal surfaces in $\mathbb{R}^3_1$.

However, flat points of timelike minimal surfaces, which are points with $K=0$ consist of umbilics and quasi-umbilics and they are not isolated in general. Hence, we can construct a counterexample to show that the same assertion as in Fact \ref{thm:Schwarz} for timelike minimal surfaces does not hold.

\begin{proposition}\label{prop:counter}
There exist two timelike minimal surfaces that are isometric each other but one of which does not belong to the congruent class of the associated family of the other.
\end{proposition}

\begin{proof}
Let us construct two flat timelike minimal surfaces that are isometric each other, one of which is a timelike plane and the other is not as follows.
\begin{align*}
f_1(u,v)&=\varphi_1(u)+{}^t(v,0,v),\quad \varphi_1(u)={}^t(-2,\sqrt{3},1)e^u/2,\\
f_2(u,v)&=\varphi_2(u)+{}^t(v,0,v),\quad \varphi_2(u)=\prescript{t\!}{}{\left(-u-\frac{3}{4}e^u,2\sqrt{3}e^{u/2}, -u+\frac{3}{4}e^u\right)},
\end{align*}
where $\varphi_1$ and $\varphi_2$ are null curves. A straightforward calculation shows that the first fundamental forms $\mathrm{I}_{f_1}$ and $\mathrm{I}_{f_2}$ of the surfaces $f_1$  and $f_2$ satisfy
\[
\mathrm{I}_{f_1}= \mathrm{I}_{f_2} =3e^ududv,
\]
and hence two surfaces $f_1$ and $f_2$ are isometric.

By definition, the surface $f_1$ is a timelike plane, which is totally umbilic. The surface $f_2$ is totally quasi-umbilic because it has a non-diagonalizable shape operator of the form
\[
\begin{pmatrix}
0 & 0 \\
\sqrt{3}e^{-u/2} & 0 
\end{pmatrix}.
\]

Finally, by Remark \ref{rema:asso_family2}, we can see that the associated family of a plane remains plane. Therefore, we obtain the desired result.
\end{proof}

Proposition \ref{prop:counter} shows that there are many flat timelike minimal surfaces. Here, we determine such flat timelike minimal surfaces. 
\begin{proposition}\label{FlatMinimal_rev}
Any flat timelike minimal surface in $\mathbb{R}^3_1$ is a cylinder whose base curve and director curve are lightlike.
\end{proposition}

\begin{proof}
We consider a local parametrization \eqref{null curves decomposition} on a domain $D$ of the form $(u_0-\varepsilon, u_0+\varepsilon)\times (v_0-\varepsilon, v_0+\varepsilon)$. On $D$, the first and the second fundamental forms can be written as follows.
\begin{center}
$\mathrm{I}=2\Lambda dudv$\quad and\quad $\mathrm{II}=Qdu^2+Rdv^2$.
\end{center}
Therefore, the shape operator $S$ and the Gaussian curvature $K$ are 
\begin{equation}\label{condition1}
\hspace{0.1cm} S=\mathrm{I}^{-1}\mathrm{II}=
\left(\begin{array}{cc} 0 & \frac{R}{\Lambda} \\ \frac{Q}{\Lambda} & 0 \\ \end{array} \right),\quad K=-\frac{QR}{\Lambda^2}.
\end{equation}
Here, we remark that the Codazzi equation shows that the coefficients $Q=Q(u)$ and $R=R(v)$ are functions of one variable, see \cite{FI}. 

When $Q$ and $R$ vanishes identically on $D$, the surface $f$ is totally umbilic, and hence it is a part of a timelike plane. 
When $Q \not\equiv 0$ or $R \not\equiv 0$ holds on $D$, without loss of generality, we may assume that $Q(u_0) \neq 0$. Since $R$ is a function of one variable and $QR\equiv 0$, we obtain $R \equiv 0$ on $D$. This means that $\psi' /\!/ \psi''$ and hence we can take a parameter $\tilde{v}=\tilde{v}(v)$ and a lightlike vector $\psi_0\in \mathbb{R}^3_1$ such that $\psi = \tilde{v}\psi_0$, which completes the proof. 
\end{proof}

\begin{remark}
The quadratic differentials $Qdu^2=\langle f_{uu}, \nu \rangle du^2$ and $Rdv^2=\langle f_{vv},\nu\rangle du^2$ are called the {\it Hopf differentials} of $f$, where $\nu$ is a unit normal vector field of $f$. Fujioka and Inoguchi \cite{FI} showed that any umbilic free timelike (not necessary minimal) surface in a Lorentzian space form satisfying the condition $Q\neq 0, R\equiv 0$ or $Q\equiv 0, R\neq 0$ must be a ruled surface whose base curve and director curve are lightlike, which is called a {\it B-scroll}. For B-scrolls in $\mathbb{R}^3_1$, see also \cite{CDM,Clelland,Graves,IT,McNertney}.

\end{remark}

By avoiding flat points, the following rigidity theorem for isometric and anti-isometric classes as in Fact \ref{thm:Schwarz} holds.

\begin{theorem}\label{thm:timelikeSchwarz}
Let $f_1$ and $f_2$ be simply connected timelike minimal surfaces in $\mathbb{R}^3_1$ without flat points and singular points. Then the following statements hold.
\begin{itemize}
\item[(1)] If $f_1$ and $f_2$ are isometric, then $f_1$ is congruent to a surface in the associated family $\{(f_2)_\theta\}_{\theta \in \mathbb{R}}$ of $f_2$.
\item[(2)] If $f_1$ and $f_2$ are anti-isometric, then $f_1$ is congruent to a surface in the associated family $\{(\hat{f_2})_\theta\}_{\theta \in \mathbb{R}}$ of the conjugate surface $\hat{f_2}$.
\end{itemize}
\end{theorem} 

To prove the theorem, it should be noted that by making an additional assumption, the following basic property that holds for holomorphic functions also holds for paraholomorphic functions. 
\begin{lemma}\label{lemma:constant}
If a paraholomorphic function $\phoro{f}$ satisfies $|\phoro{f}|^2=c$ for a non-zero constant $c$, then $\phoro{f}$ is a constant function.
\end{lemma}
This lemma is a direct consequence of the Cauchy-Riemann type equations, see \eqref{eqn:cauchyRiemann}. In the case where $c=0$, in contrast to the case of holomorphic functions, there exists a nonconstant paraholomorphic function $\phoro{f}$ satisfying $| \phoro{f}|^2=0$ such as $\phoro{f}(z)=z(1+j)$.

\begin{proof}[{\bf Proof of Theorem $\ref{thm:timelikeSchwarz}$}]
As well as surfaces in $\mathbb{R}^3$, the first, second and third fundamental forms $\mathrm{I}_f=\langle df, df \rangle$, $\mathrm{II}_f=-\langle df, d\nu \rangle$ and $\mathrm{III}_f=\langle d\nu, d\nu \rangle$ of a timelike surface $f$ satisfy the relation
\[
-K\mathrm{I}_f-2H\mathrm{II}_f+\mathrm{III}_f=0,
\]
where $K$ and $H$ are the Gaussian curvature and the mean curvature of $f$. The condition $H=0$ implies $\mathrm{III}_f=K\mathrm{I}_f$. When $f_1$ and $f_2$ are isometric (resp. anti-isometric) $\mathrm{I}_{f_1}=\mathrm{I}_{f_2}$ and $K_{f_1}=K_{f_2}$ (resp.  $\mathrm{I}_{f_1}=-\mathrm{I}_{f_2}$ and $K_{f_1}=-K_{f_2}$) hold, and hence third fundamental forms of $f_1$ and $f_2$ satisfy $\mathrm{III}_{f_1}=\mathrm{III}_{f_2}$ in both cases (1) and (2). By the definition of the third fundamental form and the assumption $K\neq 0$, it means that unit normal vector fields $\nu_1$ and $\nu_2$ of the surfaces $f_1$ and $f_2$ have the same non-degenerate first fundamental forms $\mathrm{I}_{\nu_1}=\mathrm{I}_{\nu_2}$.

On the other hand, the relation $\mathrm{I}_{\nu_1}=\mathrm{I}_{\nu_2}$ and the Weingarten equation for the surfaces $\nu_1$ and $\nu_2$ imply $\mathrm{II}_{\nu_1}=\mathrm{II}_{\nu_2}$. Therefore, by the fundamental theorem of surface theory, we conclude that $\nu_1=\nu_2$ after an isometry of $\mathbb{R}^3_1$. 

Let $(h_i,  \eta_i)$ be the Weierstrass data of the surface $f_i$ $(i=1,2)$. The relation $\nu_1=\nu_2$ shows that $h_1 = h_2$ by Remark \ref{remark:Gaussmap}. 
Since we are considering regular surfaces, the equation \eqref{eqn:wConformal} implies $|\eta_1|^2=|\eta_2 |^2 \neq 0$ when $f_1$ and $f_2$ are isometric and $|\eta_1 |^2=-|\eta_2|^2 \neq 0$ when $f_1$ and $f_2$ are anti-isometric. 
Then we obtain the paraholomorphic function 
$\hat{\eta_2}/\hat{\eta_1}$ satisfying $| \hat{\eta_2}/\hat{\eta_1}|^2=\pm 1$, where $\eta_i=\hat{\eta_i}dz^2$ $(i=1,2)$. Finally, Lemma \ref{lemma:constant} shows that $\hat{\eta_2}/\hat{\eta_1}$ is a constant function with $|\hat{\eta_2}/\hat{\eta_1}|^2=\pm 1$. Hence, there exists a real number $\theta$ such that $ \hat{\eta_2}/\hat{\eta_1} = \pm \phoro{e}^{j\theta}$ when $f_1$ and $f_2$ are isometric or $\hat{\eta_2}/\hat{\eta_1} = \pm j\phoro{e}^{j\theta}$ when $f_1$ and $f_2$ are anti-isometric.

In summary, we prove that $(h_2, \eta_2)=(h_1, \pm \phoro{e}^{j\theta}\eta_1)$ or $(h_2, \eta_2)=(h_1, \pm j\phoro{e}^{j\theta}\eta_1)$ after an isometry of $\mathbb{R}^3_1$, giving the desired result.
 
\end{proof}

\section{Deformation and symmetry}\label{sec:Goursat}

In this section, we show how symmetry of a timelike minimal surface $f$ preserved under the isometric deformation $\{f_\lambda \}_\lambda$ and the anti-isometric deformation $\{\hat{f}_\lambda \}_\lambda$. We also show that symmetry of timelike minimal surfaces is also propagated while changing its shape under more general transformations and deformations. 

We denote the paracomplex conformal group with signature $(-,+,+)$  by $\mathrm{CO}(1,2; \mathbb{C}')$, that is, 
\[
\mathrm{CO}(1,2; \mathbb{C}') = \{ A\in \mathrm{M}(3,\mathbb{C}') \mid  {}^t\! AI_{1,2}A=cI_{1,2},\ |c|^2\neq 0 \},
\]
where ${}^t\! A$ is the transposed matrix of $A$ and $I_{1,2}=\text{diag}(-1,1,1)$.

Let $f=\mathrm{Re}\int {\omega}$ be a conformal timelike minimal surface on a simply connected Lorentz surface $M$ with a paraholomorphic 1-form $\omega={}^t\!(\omega_1,\omega_2,\omega_3)$. For a matrix $A\in \mathrm{CO}(1,2; \mathbb{C}')$, the surface 
\begin{equation}\label{eq: Goursat}
f_A(p) := \mathrm{Re}\left(A\int_{p_0}^p{\omega}\right) + f_A(p_0)
\end{equation}
also gives a conformal timelike minimal surface because $A$ satisfies the condition $\langle A\omega, A\omega \rangle =c\langle \omega, \omega \rangle =0$. Based on the work \cite{G} by Goursat, we call it the {\it Goursat transformation} of $f$. 

First, we recall the following Lemma (see \cite[Lemma 5.3]{Meeks} for the Riemannian case).
 
\begin{lemma}\label{lemma:gomega}
Let $f\colon M \to \mathbb{R}^3_1$ be a simply connected timelike minimal surface represented by $f=\mathrm{Re}\left(\int{\omega}\right)$ with a paraholomorphic 1-form $\omega={}^t\!(\omega_1,\omega_2,\omega_3)$. Then
\begin{itemize}
\item[(1)] If $g\in S^\circ_f(M)$ with the linear part $O\in O(1,2)$, then $g^*\omega = O\omega$,
\item[(2)] If $g\in S^r_f(M)$ with the linear part $O\in O(1,2)$, then $g^*\omega = O\bar{\omega}$.
\end{itemize}
\end{lemma}

\begin{proof}
By the relation $f=\mathrm{Re}\left(\int{\omega}\right)$, we first remark that $\omega = df +j*df$, where $*$ is the Lorentzian Hodge star operator represented by the relation
\[
*dx = dy,\quad *dy =dx\quad
\]
for each paracomplex coordinate $z=x+jy$. By the assumption $g\in S_f(M)$, there exists an isometry $\tilde{g}$ of $\mathbb{R}^3_1$ such that $\tilde{g}\circ f = f\circ g$. Hence, the relation
\[
g^*df=g^*(f^*dx)=(f\circ g)^*dx=f^*(\tilde{g}^*dx) =f^*Odx=Of^*dx=Odf
\]
holds, where $dx={}^t\!(dx_1,dx_2, dx_3)$. 
Therefore, we obtain the desired relations (1) and (2) by using the fact that $g^*(*df)=*g^*(df)$ if $g$ is orientation preserving and $g^*(*df)=-*g^*(df)$ if $g$ is orientation reversing.
\end{proof}

\begin{remark}[For the case of surfaces with singularities]
We remark that the Hodge star operator on $1$-forms on a Lorentz surface is defined by the formula
\[
*dz=jdz,\quad *d\bar{z}=jd\bar{z}.
\]
This means that it depends only on the paracomplex structure and not on the Lorentzian metric. This fact implies that the assertions of Lemma \ref{lemma:gomega} and the results of this section hold for timelike minimal surfaces with singularities exactly the same way. Here, a singularity means a point on which the induced metric of the considered surface degenerates. For details on timelike minimal surfaces with singularities, see \cite{A,KKSY}.
\end{remark}

Here, we give a proof of Theorem \ref{thm:symmetry_Goursat_Introduction} which explains the relationships between the space groups $S_{f_A}(M)$ and $S_f(M)$.
 
 \begin{proof}[{\bf Proof of Theorem \ref{thm:symmetry_Goursat_Introduction}}]
 We prove only the case (1). Suppose $f_A=\mathrm{Re}\left(A\int{\omega}\right)$ and $g\in S_f^\circ(M)$ has linear part $O$. To prove the sufficiency let us assume that $AO=\widetilde{O}A$.
 
  First, we prove that $g\in S_f^\circ(M)$ is an isometry of $f_A(M)$. The first fundamental form $\mathrm{I}_{f_A}$ of $f_A$ is written by 
 \[
 \mathrm{I}_{f_A}= 4\langle (f_A)_w, \overline{(f_A)_w} \rangle dwd\bar{w} = \langle A\omega, \overline{A\omega} \rangle,
 \]
 and hence Lemma \ref{lemma:gomega} and the assumption $AO=\widetilde{O}A$ imply that
 \[
 g^* \mathrm{I}_{f_A} =  \langle Ag^*\omega, \overline{A g^*\omega} \rangle 
= \langle AO\omega, \overline{A O\omega} \rangle 
= \langle \widetilde{O}A\omega, \overline{\widetilde{O}A\omega} \rangle 
= \langle A\omega, \overline{A\omega} \rangle =\mathrm{I}_{f_A},
 \]
 which means that $g$ is also an isometry of $f_A(M)$.
 
 Next, we check $g$ also induces a symmetry of $f_A(M)$.
  \begin{align*}
f_A(g(p)) &= \mathrm{Re}\left(A\int_{p_0}^{g(p)}{\omega}\right) + f_A(p_0) \\
&= \mathrm{Re}\left(A\int_{p_0}^{g(p_0)}{\omega}\right) + \mathrm{Re}\left(A\int_{g(p_0)}^{g(p)}{\omega}\right)+ f_A(p_0) \\
&= \mathrm{Re}\left(A\int_{p_0}^{p}{g^*\omega}\right)+ f_A(g(p_0)) \\
&= \mathrm{Re}\left(A\int_{p_0}^{p}{O\omega}\right)+ f_A(g(p_0)) \\
&= \mathrm{Re}\left(AO\int_{p_0}^{p}{\omega}\right)+ f_A(g(p_0)) \\
&= \mathrm{Re}\left(\widetilde{O}A\int_{p_0}^{p}{\omega}\right)+ f_A(g(p_0)) \\
&= \widetilde{O}f_A(p)- \widetilde{O}f_A(p_0)+ f_A(g(p_0)). 
 \end{align*}
Conversely, if $g\in S_{f_A}(M)$ and its linear part is a matrix $\widetilde{O}\in \mathrm{O}(1,2)$. Then there exists a vector $\tilde{t}\in \mathbb{R}^3_1$ such that $f_A\circ g = \widetilde{O}f_A + \tilde{t}$. Based on the relation $g^*(A\omega)=AO\omega$, by taking the derivative of this equation, we have $AO=\widetilde{O}A$.

 This proves (1) and the case (2) is proved similarly by using the assumption $AO=\widetilde{O}\bar{A}$ and the relation $g^*\omega =O\bar{\omega}$ in Lemma \ref{lemma:gomega}.
 \end{proof}
 
 Since the identity matrix $I_3$ commutes with arbitrary matrix, we have the following result regarding the conservation of translation symmetry.
 
 \begin{corollary}\label{cor:translation}
 Under the same assumptions as in Theorem \ref{thm:symmetry_Goursat_Introduction}, suppose an orientation preserving isometry $g\in S^\circ_f(M)$ gives a translation symmetry of a surface $f$, then $g\in S^\circ_{f_A}(M)$ also gives a translation symmetry of arbitrary Goursat transformation $f_A$ for $A\in \mathrm{CO}(1,2; \mathbb{C}')$ whenever the translation vector does not vanish.
 \end{corollary}
 By using this corollary, we can produce many periodic timelike minimal surfaces as Example \ref{example:Bonnet}.
 
 \begin{remark}
In \cite[Corollary 5.2]{Meeks}, Meeks pointed out a necessary and sufficient condition for the conjugate minimal surface in the Euclidean space to have a translation symmetry. 
Moreover, Leschke and Moriya \cite{LM} revealed similar results on the conservation of translation symmetry of simple factor dressing of minimal surfaces in $\mathbb{R}^3$ and $\mathbb{R}^4$, which is also a special kind of Goursat transformations. See Theorem 6.2 and Corollary 6.8 in \cite{LM}.
\end{remark}

Goursat transformations of the form \eqref{eq: Goursat} include various important transformations and deformations of timelike minimal surfaces. 
 From now on, we give some applications of Theorem \ref{thm:symmetry_Goursat_Introduction} for specific deformations.

\subsection{Associated family and conjugation, revisited}

The isometric deformation $\{f_\theta\}_{\theta \in \mathbb{R}}$ and the anti-isometric deformation $\{\hat{f}_\theta\}_{\theta \in \mathbb{R}}$ in Definition \ref{def:associated} are also corresponding to the Goursat transformations for the matrices $\phoro{e}^{j\theta}I_3$ and $j\phoro{e}^{j\theta}I_3$, respectively. Therefore, we obtain the following result, which is a timelike counter part of the result for minimal surfaces by Meeks \cite[Theorem 5.5]{Meeks}, describing how space groups behave with respect to the above transformations.

\begin{corollary}\label{cor:Spacegroup}
Let $f\colon M \to \mathbb{R}^{3}_1$ be a simply connected timelike minimal surface. Then 
\begin{itemize}
\item[(1)] $S_f^\circ(M) = S_{f_\theta}^\circ(M)=  S_{\hat{f}_\theta}^\circ(M)$ for all $\theta \in \mathbb{R}$. Moreover, the linear part of $g\in S_{f}^\circ(M)$ is preserved under the isometric deformation $\{f_\theta\}_{\theta\in \mathbb{R}}$ and the anti-isometric deformation $\{\hat{f}_\theta\}_{\theta\in \mathbb{R}}$.
\item[(2)] If $g\in S_f^r(M)$, then $g\not \in S_{f_\theta}^r(M)$ for all $\theta\neq 0$.
\item[(3)] If $g\in S_f^r(M)$, then $g\in S_{\hat{f}_\theta}^r(M)$ for some $\theta \in \mathbb{R}$ if and only if $\theta =0$. Moreover, if $g\in S_f^r(M)$ has the linear part $O$, then $g\in S_{\hat{f}}^r(M)$ has the linear part $-O$.
\end{itemize}
\end{corollary}

\begin{proof}
Let $g$ be a symmetry in $S_f(M)$ whose linear part is $O$.

 When $g$ is orientation preserving, the assertion $(1)$ follows from Theorem \ref{thm:symmetry_Goursat_Introduction} and the fact that $\phoro{e}^{j\theta}I_3$ and $j\phoro{e}^{j\theta}I_3$ commute with $O$.
 
 When $g$ is orientation reversing, Theorem \ref{thm:symmetry_Goursat_Introduction} asserts that $g\in S_{f_\theta}(M)$ holds if and only if there exists $\tilde{O}\in \mathrm{O}(1,2)$ such that 
\[
\phoro{e}^{j\theta}I_3O=\widetilde{O}\overline{\phoro{e}^{j\theta}I_3} \Leftrightarrow \phoro{e}^{2j\theta}O = \widetilde{O}.
\]
Therefore, we obtain $\theta =0$ proving the assertion (2). On the other hand, $g\in S_{\hat{f}_\theta}(M)$ holds if and only if there exists $\tilde{O}\in \mathrm{O}(1,2)$ such that 
\[
j\phoro{e}^{j\theta}I_3O=\widetilde{O}\overline{j\phoro{e}^{j\theta}I_3} \Leftrightarrow \phoro{e}^{2j\theta}O =-\widetilde{O}.
\]
Therefore, we obtain $\theta=0$ and $\widetilde{O}=-O$ proving the assertion (3).
\end{proof}

From now on, we focus on the conjugation. Taking the conjugation $\hat{f}$ of a timelike minimal surface $f$ corresponds to the Goursat transformation of $f$ with respect to the matrix
\begin{equation}\label{Conjugate_matrix}
J = \begin{pmatrix}
j & 0 & 0 \\
0 & j & 0 \\
0 & 0 & j
\end{pmatrix}
\in \mathrm{CO}(1,2; \mathbb{C}').
\end{equation}
Since $J$ commutes with arbitrary matrix, conjugation and arbitrary Goursat transformation commute up to a constant vector. More precisely, the following commutative diagram holds.
\[
  \begin{diagram}
    \node[2]{f}
    \arrow[3]{e,b,2}{\text{Goursat transformation}}
    \arrow[1]{s,l}{\text{conjugation}}
    \node[3]{f_A}
    \arrow[1]{s,r}{\text{conjugation}} \\
    \node[2]{\hspace{1cm}\hat{f}=f_J\hspace{1cm}} 
     \arrow[3]{e,t,1}{\hspace{1cm}\text{Goursat transformation}}
    \node[3]{\hspace{1cm} \hat{f}_A=f_{AJ}=f_{JA}}
  \end{diagram}
\]
where $A\in \mathrm{CO}(1,2; \mathbb{C}')$. 
Due to the relation $\bar{J}=-J$, Theorem \ref{thm:symmetry_Goursat_Introduction} and Corollary \ref{cor:Spacegroup}, we obtain the following relation between space groups of the surfaces $f, f_A, \hat{f}=f_J$ and $\hat{f}_A=f_{AJ}$.

 \begin{corollary}\label{cor:diagram}
  Let $f\colon M\to \mathbb{R}^3_1$ be a simply connected timelike minimal surface, $g$ be an orientation preserving (resp.~reversing) isometry of $f(M)$. Assume that $A$ be a matrix in $ \mathrm{CO}(1,2; \mathbb{C}')$ and $O, \widetilde{O}$ be matrices in $\mathrm{O}(1,2)$ such that $AO=\widetilde{O}A$ (resp.~ $AO=\widetilde{O}\bar{A}$). Then the following statements are equivalent.
   \begin{itemize}
    \item[(1)] $g\in S_f(M)$ with linear part $O$,
    \item[(2)] $g\in S_{\hat{f}}(M)$ with linear part $O$ (resp. $-O$),
    \item[(3)] $g\in S_{f_A}(M)$ with linear part $\widetilde{O}$,
    \item[(4)] $g\in S_{\hat{f}_A}(M)$ with linear part $\widetilde{O}$ (resp. $-\widetilde{O}$).
  \end{itemize}
  \end{corollary}
  
 
\subsection{Self duality relation}\label{sec:duality}
The next important Goursat transformation is the following self duality relation.  We call the Goursat transformation $f_D$ of a timelike minimal surface $f$ with the matrix
\begin{equation}
 D=\begin{pmatrix}\label{Dmatrix}
j & 0 & 0 \\
0 & 1 & 0  \\
0 & 0 & 1
\end{pmatrix}
\in \mathrm{CO}(1,2; \mathbb{C}')
\end{equation}
the {\it dual timelike minimal surface} of $f$.

The dual timelike minimal surface $f_D$ has the following notable property.

\begin{proposition}\label{prop: dual data}
Let $f$ be a timelike minimal surface with the Weierstrass data $(h,\eta)$. The Weierstrass data $(h_D, \eta_D)$ of the dual timelike minimal surface $f_D$ is 
\[
h_D=\cfrac{\left(1+j\right)h-\left(1-j\right)1/h}{2},\quad \eta_D=\cfrac{\left(1+j\right)1/h-\left(1-j\right)h}{2}\, h\eta.
\]
In particular, the duality reverses the signs of the Gaussian curvatures $K_f$ of the surface $f$ and $K_{f_D}$ of the dual $f_D$ as follows
\[
\mathrm{sgn}(K_f) = - \mathrm{sgn}(K_{f_D}).
\] 
\end{proposition}

\begin{proof}
The former relations follow immediately from a straightforward calculation. Since the second fundamental form $\mathrm{II}_f$ of $f$ is written as 
 \[
 \mathrm{II}_f=\Re{(\eta dh)},
 \]
 we can check the latter property by showing the relation $\eta_Ddh_D =j\eta dh$. 
\end{proof}

\begin{remark}[Duality between minimal surfaces and maximal surfaces]\label{maximalduality}
Let us consider a minimal surface in $\mathbb{R}^3$ written as $f=\Re{\prescript{t\!}{}{\left(\omega_1, \omega_2, \omega_3 \right)}}$ with holomorphic one forms $\omega_j$ ($j=1,2,3$), and the transformation
\[
\begin{pmatrix}
 \tilde{\omega}_1 \\
\tilde{\omega}_2  \\
\tilde{\omega}_3
\end{pmatrix} 
=\widetilde{D}
\begin{pmatrix}
 \omega_1 \\
\omega_2  \\
\omega_3
\end{pmatrix},
\quad \widetilde{D}:= \begin{pmatrix}
i & 0 & 0 \\
0 & 1 & 0  \\
0 & 0 & 1
\end{pmatrix}
\]
where $i$ is the imaginary unit on the complex plane $\mathbb{C}$ satisfying $i^2=-1$. Although the matrix $\widetilde{D}$ belongs neither to the complex orthogonal group $\mathrm{O}(3; \mathbb{C})=\{ A\in \mathrm{M}(3,\mathbb{C}) \mid  {}^t\! AA=I_{3} \},$ nor to the indefinite complex orthogonal group $\mathrm{O}(1,2; \mathbb{C})= \{ A\in \mathrm{M}(3,\mathbb{C}) \mid  {}^t\! AI_{1,2}A=I_{1,2} \}$, the surface $f_{\tilde{D}}:=\Re{\prescript{t\!}{}{\left(\tilde{\omega}_1, \tilde{\omega}_2, \tilde{\omega}_3 \right)}}$ gives a maximal surface in $\mathbb{R}^3_1$. 
This one to one correspondence between minimal surfaces in $\mathbb{R}^3$ and maximal surfaces in $\mathbb{R}^3_1$ is called the {\it duality}, see \cite{Lee} and also \cite{AF,AL,LLS,UY} for example. 

We also remark that any minimal surface has non positive Gaussian curvature $K\leq 0$ and any maximal surface has non negative Gaussian curvature $K\geq 0$. Therefore, Proposition \ref{prop: dual data} means that the self duality between $f$ and $f_D$ is a Lorentizan version of the above duality between minimal and maximal surfaces.
\end{remark}

At the end of this subsection, we discuss symmetries derived from reflection principles, which are closely related to the conjugation $\hat{f}=f_J$ and the dual $f_D$. 
It is well known that if a minimal surface in Euclidean space has a straight line, then the surface has a symmetry with respect to the line, and its conjugate surface has a planar symmetry with respect to a plane orthogonal to the line. The same results also valid for maximal surfaces and timelike minimal surfaces in $\mathbb{R}^3_1$, see \cite{{ACM1}} and \cite{KKSY} for example. Since such a reflection symmetry is obtained by an orientation reversing isometry of the form $g(z)=\bar{z}$ for an appropriate coordinate $z$, the equivalence between (1) and (2) in Corollary \ref{cor:diagram} is a generalization of this fact. 
It can be also used to express the relation in the correspondence between the symmetries with respect to a shrinking singularity on a timelike minimal surface $f$ and a folding singularity on the conjugate surface $\hat{f}$.

To deal with singularities on timelike minimal surfaces, we recall the classes of generalized timelike minimal surfaces introduced in \cite{KKSY}.  

A non-constant smooth map $f\colon M \longrightarrow \mathbb{R}^3_1$ from a Lorentz surface $M$ into $\mathbb{R}^3_1$ is called a {\it generalized timelike minimal surface} if $f$ is immersed on an open dense set of $M$ and there exists a local coordinate system $(U; x, y)$ near each point of $M$ such that $\langle f_x, f_x \rangle = -\langle f_y, f_y \rangle$, $\langle f_x, f_y \rangle \equiv0$ and $f_{xx}-f_{yy}\equiv 0$ on $U$. 
For each local coordinate system $(U; x, y)$, let
\[
		\mathcal{A}=\{p\in U \mid \text{$f_x(p)$ or $f_y(p)$ is lightlike in $\mathbb{R}^3_1$} \},\quad \mathcal{B}=\{p\in U\mid d{f}_p=0\}.
\]
Since the induced metric $\mathrm{I}_f$ degenerates at each point $p$ in $\mathcal{A}\cup \mathcal{B}$, we call $p$ a {\it singular point} of $f$. A singular point $p \in \mathcal{A}$ is called a {\it shrinking singular point} (or a {\it conelike singular point}) if there is a regular curve $\gamma\colon I \to U$ from an interval $I$ passing through $p$ such that $\gamma(I)\subset \mathcal{A}$ and $f\circ \gamma (I)$ becomes a single point in $\mathbb{R}^3_1$, which we call a {\it shrinking singularity}.
Also a singular point $p \in \mathcal{A}$ is called a {\it folding singular point} (or a {\it fold singular point}) if there is a neighborhood of $p$ on which the surface is reparametrized as $p=(0,0)$ and $f_y(x,0)\equiv 0$. We call the image $\{f(x,0)\mid (u,0)\in U)\}$ a {\it folding singularity}. 
By using the singular Bj\"oling representation formula, reflection principles with respect to shrinking singular points and folding singular points have been proved in \cite[Lemma 4.3 and 4.5]{KKSY}.  

By considering the Goursat transformations $\hat{f}=f_J$ and $f_D$, symmetries about lines, planes, shrinking singularities, and folding singularities can be unified as follows (the same result holds for Riemannian case, see \cite{AF} for more details).
\begin{corollary}\label{cor:quadruple}
Let $f\colon M\to \mathbb{R}^3_1$ be a simply connected timelike minimal surface. Then the following statements are equivalent.
   \begin{itemize}
    \item[(1)] $f$ has the line symmetry with respect to a timelike straight line on the surface $f$ which is parallel to the $x_1$-axis,  
    \item[(2)] $\hat{f}=f_J$ has the planar symmetry with respect to a spacelike plane  parallel to the $x_2x_3$-plane which is perpendicular to the surface $\hat{f}$,
    \item[(3)] $f_D$  has the point symmetry with respect to a shrinking singularity, and
    \item[(4)] $\hat{f}_D=f_{DJ}=f_{JD}$ has the folded symmetry with respect to a folding singularity.
 \end{itemize}
 \begin{proof}
We give only an argument from the case where $f$ has the straight line of the form $f(x+j0)={}^t(x,0,0)$. By the regular reflection principle in \cite[Lemma 4.1]{KKSY}, we obtain
\[
f(\bar{z}) = Of(z),\quad \text{where $O=\text{diag}(1, -1, -1)$ and  $z=x+jy$}.
\]
By the relation $JO=-O\overline{J}$ and Corollary \ref{cor:diagram}, the conjugate surface $f_J$ has the symmetry $f_J(\bar{z})=-Of_J(z)$ up to a translation which proving the assertion (2).  Similarly, the relation $DO=-I_3\overline{D}$ induces the point symmetry $f_D(\bar{z})=-f_D(z)$ up to a translation. In particular,  $f_D(x+j0)$ shrinks to a single point in $\mathbb{R}^3_1$ proving the assertion (3). Finally, the relation $D(-O)=I_3\overline{D}$ induces the folded symmetry $f_{JD}(\bar{z})=f_{JD}(z)$ up to a translation proving the assertion (4).
 \end{proof}
\end{corollary}


\subsection{L\'opez-Ros deformation}
Another interesting example of Goursat transformations is the following L\'opez-Ros type deformation.
 We define the {\it L\'opez-Ros deformation} $\{f_\lambda \}_{\lambda>0}$ of a timelike minimal surface $f$ by changing Weierstrass data from $(h, \eta)$ to $(\lambda h, \eta / \lambda)$, that is,
 \begin{equation}\label{eq:Gourast}
 f_\lambda(p)  
  =\mathrm{Re}\int^p_{p_0}\prescript{t\!}{}{\left( -\left(\frac{1}{\lambda}+\lambda h^2\right), j\left(\frac{1}{\lambda}-\lambda h^2\right), 2h \right)}\eta
    +f_\lambda(p_0). 
 \end{equation}
 This deformation was introduced in \cite{LR} for minimal surfaces in $\mathbb{R}^3$, and the deformation $\{f_\lambda \}_{\lambda>0}$ preserves the second fundamental form as the original L\'opez-Ros deformations since the second fundamental form $\mathrm{II}_f$ of $f=f_1$ is written as 
 \[
 \mathrm{II}_f=\Re{(\eta dh)}.
 \]
 
 Moreover, a straightforward calculation shows that the deformation \eqref{eq:Gourast} is obtained by the Goursat transformation of $f=f_1$ with respect to the matrix
\begin{equation}\label{LRmatrix}
A=A(\lambda) = \begin{pmatrix}
\frac{\lambda+\frac{1}{\lambda}}{2} & j\frac{\lambda-\frac{1}{\lambda}}{2} & 0 \\
j\frac{\lambda-\frac{1}{\lambda}}{2} & \frac{\lambda+\frac{1}{\lambda}}{2} & 0 \\
0 & 0 & 1
\end{pmatrix}
\in \mathrm{CO}(1,2; \mathbb{C}').
\end{equation}
Whereas the isometric deformation $\{f_\lambda \}_\lambda$ preserving the first fundamental form $\mathrm{I}_f$ and the anti isometric deformation $\{\hat{f}_\lambda\}_\lambda$ preserving $-\mathrm{I}_f$ have preserved a kind of symmetries, as an application of Theorem \ref{thm:symmetry_Goursat_Introduction}, we can control symmetries of timelike minimal surfaces while keeping the second fundamental form $\mathrm{II}_f$ via the L\'opez-Ros deformation.

\begin{remark}[Ambient isometry as a Goursat transformation]
We should remark that ambient isometries in $\mathbb{R}^3_1$ and Goursat transformations \eqref{eq: Goursat} do not commute in general, and this noncommutativity produces different surfaces which are not isometric to the original surface.  
In the case of the duality of minimal and maximal surfaces, such noncommutativity was discussed by Ara\'ujo and Leite \cite{AL} (see also Remark \ref{maximalduality}). 

In the case of timelike minimal surfaces, up to a translation, we can see an ambient isometry in $\mathbb{R}^3_1$ as a Goursat transformation with a matrix in $\mathrm{O}(1,2)\subset \mathrm{CO}(1,2; \mathbb{C})$. Thus, we can also handle by Theorem \ref{thm:symmetry_Goursat_Introduction} the symmetries of different surfaces that arise from the noncommutativity of these transformations.
\end{remark}

\section{Examples}\label{sec:Ex}
In this section, we give concrete examples of symmetry relations discussed in the previous section.
Let us first see how Corollary \ref{cor:diagram} gives many symmetry relations between $f, f_D, \hat{f}$ and $\hat{f}_D$.

\begin{example}[Lorentzian Enneper surface, parabolic helicoid and their conjugates]\label{example:Enneper}
Let us take the Weierstrass data $(h,\eta)=(z, dz)$ defined on $\mathbb{C}'$. The surface written by \eqref{eq:pW} is
\begin{align*}
f(z=x+jy) &=\mathrm{Re}\prescript{t\!}{}{\left( -z-\frac{z^3}{3}, j\left(z-\frac{z^3}{3}\right), z^2 \right)}\\
&=\prescript{t\!}{}{\left(-x-\frac{x^3}{3}-xy^2, y-x^2y-\frac{y^3}{3},x^2+y^2 \right)}
\end{align*}
and it is called the {\it Lorentzian Enneper surface} (see \cite{Konderak} for example). By using the matrix $D$ in \eqref{Dmatrix}, its dual $f_D$ is written as follows.
\begin{align*}
f_D (z) &= \mathrm{Re}\left(D \prescript{t\!}{}{\left( -z-\frac{z^3}{3}, j\left(z-\frac{z^3}{3}\right), z^2 \right)}\right)\\
&=\prescript{t\!}{}{\left(-y-x^2y-\frac{y^3}{3}, y-x^2y-\frac{y^3}{3},x^2+y^2 \right)}.
\end{align*}
The surface $f_D$ and its conjugate $\hat{f}_D$ are nothing but the surfaces called the {\it timelike parabolic helicoid} and the {\it timelike parabolic catenoid}, respectively (see \cite{KKSY} for example). Let us see how the symmetries of $f, f_D, \hat{f}$ and $\hat{f}_D$ relate to each other. 

First, the surface $f$ has the following orientation reversing planar symmetries: 
\begin{align*}
f(\bar{z}) = \begin{pmatrix}
1 & 0 & 0 \\
0 & -1 & 0  \\
0 & 0 & 1
\end{pmatrix}f(z),\quad 
f(-\bar{z}) = \begin{pmatrix}
-1 & 0 & 0 \\
0 & 1 & 0  \\
0 & 0 & 1
\end{pmatrix}f(z).
\end{align*}

By Theorem \ref{thm:symmetry_Goursat_Introduction} and the matrix relations
\begin{align*}
D\begin{pmatrix}
1 & 0 & 0 \\
0 & -1 & 0  \\
0 & 0 & 1
\end{pmatrix}
=
\begin{pmatrix}
-1 & 0 & 0 \\
0 & -1 & 0  \\
0 & 0 & 1
\end{pmatrix}
\overline{D}, \qquad 
D\begin{pmatrix}
-1 & 0 & 0 \\
0 & 1 & 0  \\
0 & 0 & 1
\end{pmatrix}
=
\begin{pmatrix}
1 & 0 & 0 \\
0 & 1 & 0  \\
0 & 0 & 1
\end{pmatrix}\overline{D},
\end{align*}
the dual timelike minimal surface $f_D$ has the following symmetries:
\begin{align*}
f_D(\bar{z}) = \begin{pmatrix}
-1 & 0 & 0 \\
0 & -1 & 0  \\
0 & 0 & 1
\end{pmatrix}f_D(z),\quad 
f_D(-\bar{z}) = \begin{pmatrix}
1 & 0 & 0 \\
0 & 1 & 0  \\
0 & 0 & 1
\end{pmatrix}f_D(z)
\end{align*}
which mean that $f_D$ has the line symmetry with respect to the spacelike $x_3$-axis and the folded symmetry with respect to fold singularities along $\Im{z}=0$, see the top right of Figure \ref{Fig:Enneper}.

Furthermore, by Corollary \ref{cor:diagram}, the conjugate surface $\hat{f}$ has the following line symmetries with respect to the spacelike $x_2$-axis and the timelike $x_1$-axis:
\begin{align*}
\hat{f}(\bar{z}) = \begin{pmatrix}
-1 & 0 & 0 \\
0 & 1 & 0  \\
0 & 0 & -1
\end{pmatrix}\hat{f}(z),\quad 
\hat{f}(-\bar{z}) = \begin{pmatrix}
1 & 0 & 0 \\
0 & -1 & 0  \\
0 & 0 & -1
\end{pmatrix}\hat{f}(z).
\end{align*}

Finally, Corollary \ref{cor:diagram} shows that $\hat{f}_D$ has the following symmetries:
\begin{align*}
\hat{f}_D(\bar{z}) = \begin{pmatrix}
1 & 0 & 0 \\
0 & 1 & 0  \\
0 & 0 & -1
\end{pmatrix}\hat{f}_D(z),\quad 
\hat{f}_D(-\bar{z}) = \begin{pmatrix}
-1 & 0 & 0 \\
0 & -1 & 0  \\
0 & 0 & -1
\end{pmatrix}\hat{f}_D(z).
\end{align*}
which mean that $\hat{f}_D$ has the planar symmetry with respect to the timelike $x_1x_2$-plane and the point symmetry with respect to shrinking singularities along $\Im{z}=0$, see the bottom right of Figure \ref{Fig:Enneper}.

For orientation preserving isometries, we can also check that by Corollary \ref{cor:diagram} $f, f_D, \hat{f}$ and $\hat{f}_D$ share the line symmetry with respect to the spacelike $x_3$-axis:

\begin{align*}
f(-z) = \begin{pmatrix}
-1 & 0 & 0 \\
0 & -1 & 0  \\
0 & 0 & 1
\end{pmatrix}f(z),\quad 
f_D(-z) = \begin{pmatrix}
-1 & 0 & 0 \\
0 & -1 & 0  \\
0 & 0 & 1
\end{pmatrix}f_D(z).\\
\hat{f}(-z) = \begin{pmatrix}
-1 & 0 & 0 \\
0 & -1 & 0  \\
0 & 0 & 1
\end{pmatrix}\hat{f}(z),\quad 
\hat{f}_D(-z) = \begin{pmatrix}
-1 & 0 & 0 \\
0 & -1 & 0  \\
0 & 0 & 1
\end{pmatrix}\hat{f}_D(z).
\end{align*}
\end{example}

\begin{figure}[!h]
\hspace{+3.5cm}
\vspace{-1.5cm}
\begin{center}
\includegraphics[clip,scale=0.45,bb=0 0 600 600]{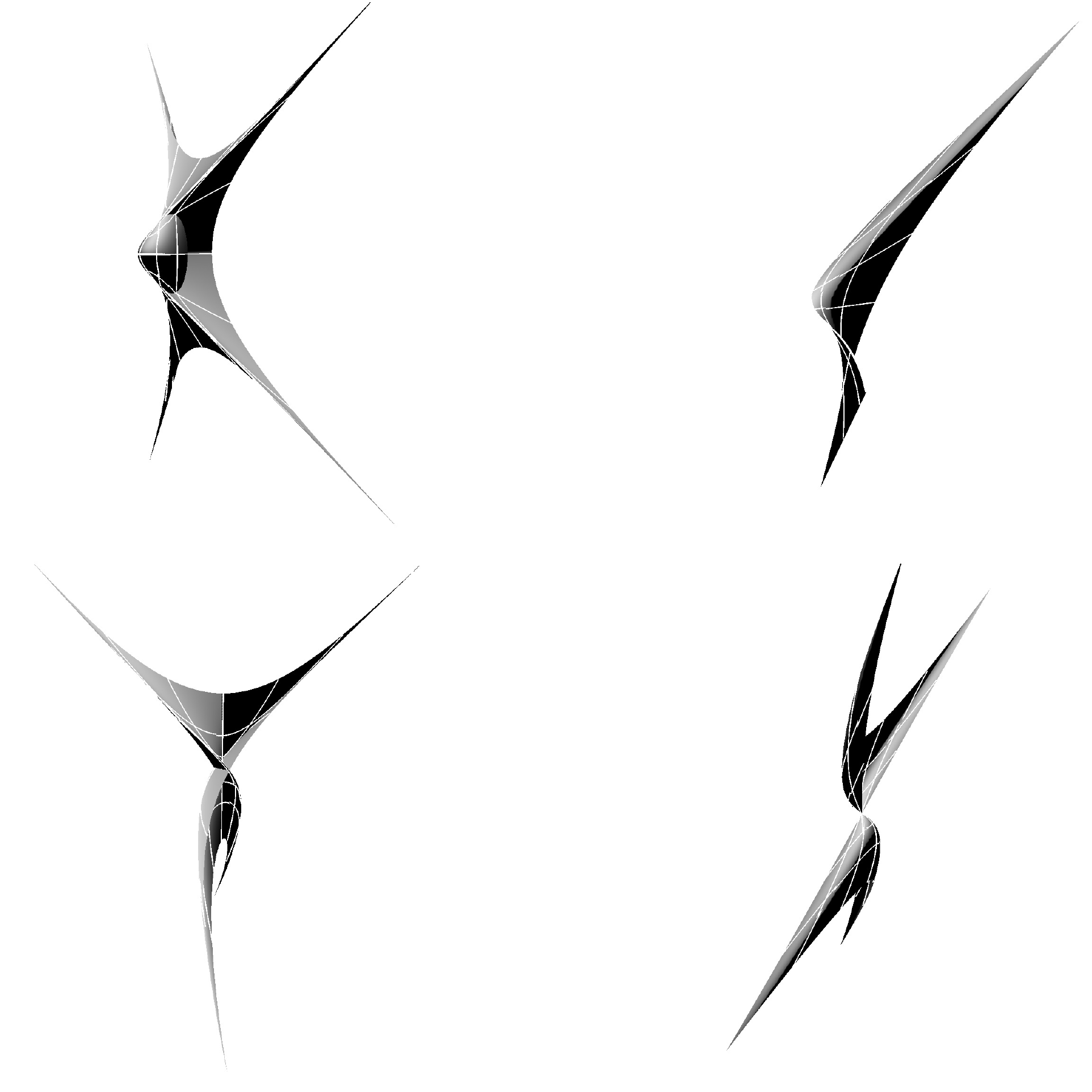}
\end{center}
\caption{The Lorentzian Enneper surface $f$ (left top), the parabolic helicoid $f_D$ (right top), the conjugate Enneper surface $\hat{f}$ (left bottom) and the parabolic catenoid $\hat{f}_D$ (right bottom).}\label{Fig:Enneper}
\end{figure}

\begin{example}[Periodic Bonnet type surfaces]\label{example:Bonnet}
Finally, we see symmetry relations for L\'opez-Ros deformation. 
Let us take the Weierstrass data $(h,\eta)=(\ptan{z}, \frac{1}{2}\pcos^2{z}dz)$ defined on $\mathbb{C}'$. The surface written by \eqref{eq:pW} is
\begin{align*}
f(z=x+jy) &=\mathrm{Re}\prescript{t\!}{}{\left( -\frac{z}{2}, \frac{j}{4}\psin{2z}, -\frac{1}{4}\pcos{2z} \right)}\\
&=\prescript{t\!}{}{\left(-\frac{x}{2}, \frac{1}{4}\cos{2x}\sin{2y}, -\frac{1}{4}\cos{2x}\cos{2y} \right)}
\end{align*}
and it is called the {\it elliptic catenoid}, which is a rotational timelike minimal surface. By using the matrix $A(\lambda)$ in \eqref{LRmatrix}, its L\'opez-Ros deformation $f_\lambda=f_{A(\lambda)}$ is as follows
\begin{align*}
f_\lambda (z) &= \mathrm{Re}\left(A(\lambda) \prescript{t\!}{}{\left( -\frac{z}{2}, \frac{j}{4}\psin{2z}, \frac{1}{4}\pcos{2z} \right)} \right). 
\end{align*}
The Weierstrass data of $f_\lambda$ is $(\lambda h,\eta/\lambda)=(\lambda \ptan{z}, \frac{1}{2\lambda}\pcos^2{z}dz)$, and hence the surface $f_\lambda$ is exactly the surface known as a {\it timelike minimal Bonnet type surface}, on which each curvature line lies on a plane for any $\lambda>0$, see Figure \ref{Fig:Bonnet}. For more details of such surfaces, see \cite{ACO}. 
Let us see how the symmetry of $f$ is preserved or changed via the L\'opez-Ros deformation $\{f_\lambda\}_{\lambda>0}$.

The surface $f$ has the following orientation preserving symmetries: 
\begin{align*}
f(z+\pi) = f(z) +\prescript{t\!}{}{\left( -\frac{\pi}{2},0, 0 \right)},\quad
f(z+j\pi) = f(z),\quad
f(-z) = \begin{pmatrix}
-1 & 0 & 0 \\
0 & -1 & 0  \\
0 & 0 & 1
\end{pmatrix}f(z).
\end{align*}
Obviously, the linear parts of these symmetries commute with the matrix $A(\lambda)$ in \eqref{LRmatrix} for each $\lambda$. Hence, Theorem \ref{thm:symmetry_Goursat_Introduction} implies that the above symmetries are propagated to the deformed surface $f_\lambda$. In particular, the linear parts of the above symmetries are preserved as follows.
\begin{align*}
&f_\lambda(z+\pi) = f_\lambda(z) + \prescript{t\!}{}{\left( -\frac{\pi}{4}\left(\lambda+\frac{1}{\lambda}\right),0, 0 \right)},\quad \\
f_\lambda(z+j\pi) &= f_\lambda(z)+\prescript{t\!}{}{\left(0, -\frac{\pi}{4}\left(\lambda-\frac{1}{\lambda}\right), 0 \right)},\quad 
f_\lambda(-z) = \begin{pmatrix}
-1 & 0 & 0 \\
0 & -1 & 0  \\
0 & 0 & 1
\end{pmatrix}f_\lambda(z).
\end{align*}
These relations mean that the surface $f_\lambda$ is doubly periodic for $\lambda \neq 1$ and the elliptic catenoid $f=f_1$ is singly periodic, and $f_\lambda$ for any $\lambda>0$ has the line symmetry with respect to $x_3$-axis.

Also, $f$ has the following orientation reversing symmetries:
\[
f(\bar{z}) = \begin{pmatrix}
1 & 0 & 0 \\
0 & -1 & 0  \\
0 & 0 & 1
\end{pmatrix}f(z),\quad 
f(-\bar{z}) = \begin{pmatrix}
-1 & 0 & 0 \\
0 & 1 & 0  \\
0 & 0 & 1
\end{pmatrix}f(z).
\]
If $O$ be any of the above matrices, we can check that $O$ satisfies $AO=O\bar{A}$ for the matrix $A=A(\lambda)$ in \eqref{LRmatrix} for each $\lambda$. Hence, Theorem \ref{thm:symmetry_Goursat_Introduction} implies that the above symmetries are propagated to the deformed surface $f_\lambda$ as follows.
\[
f_\lambda(\bar{z}) = \begin{pmatrix}
1 & 0 & 0 \\
0 & -1 & 0  \\
0 & 0 & 1
\end{pmatrix}f_\lambda(z),\quad 
f_\lambda(-\bar{z}) = \begin{pmatrix}
-1 & 0 & 0 \\
0 & 1 & 0  \\
0 & 0 & 1
\end{pmatrix}f_\lambda(z),
\]
which mean that  the planar symmetries with respect to the timelike $x_1x_3$-plane and the spacelike $x_2x_3$-plane on the surface $f_\lambda$ are preserved for any $\lambda>0$.

\vspace{-1.3cm}
\begin{figure}[!h]
\begin{center}
\begin{tabular}{c}
\hspace{+1.5cm}
\begin{minipage}{0.5\hsize}
\begin{center}
\vspace{-1.8cm}
\includegraphics[clip,scale=0.45,bb=0 0 500 509]{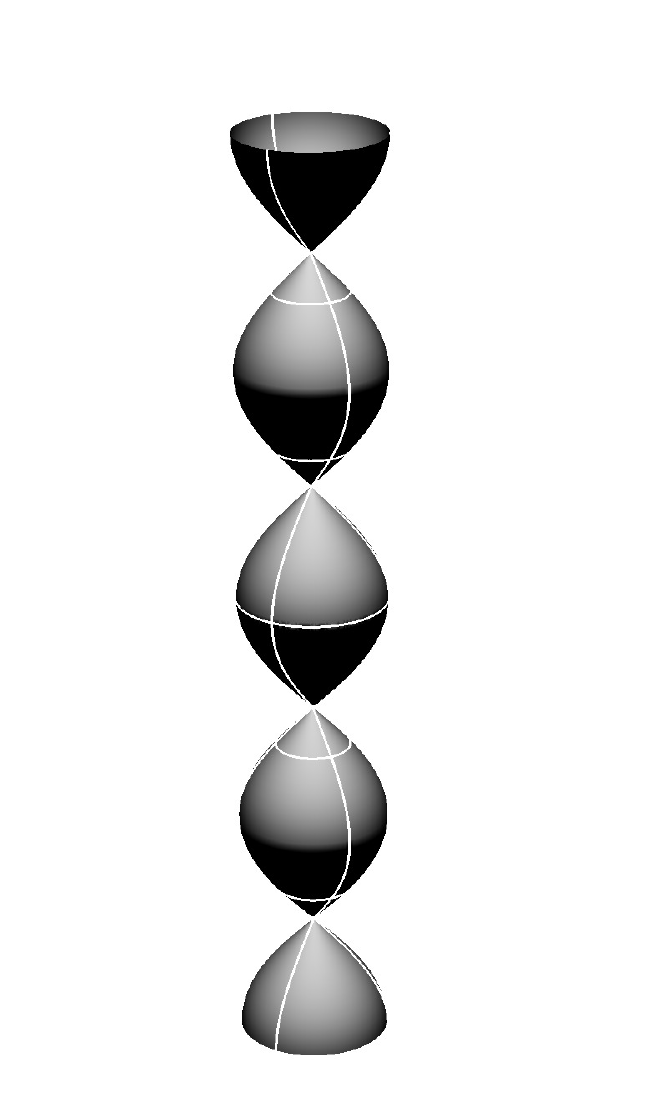}
\vspace{0.5cm}
\end{center}
\end{minipage}
\hspace{-4.0cm}
\begin{minipage}{0.5\hsize}
\begin{center}
\vspace{-1.1cm}
\includegraphics[clip,scale=0.45,bb=0 0 555 449]{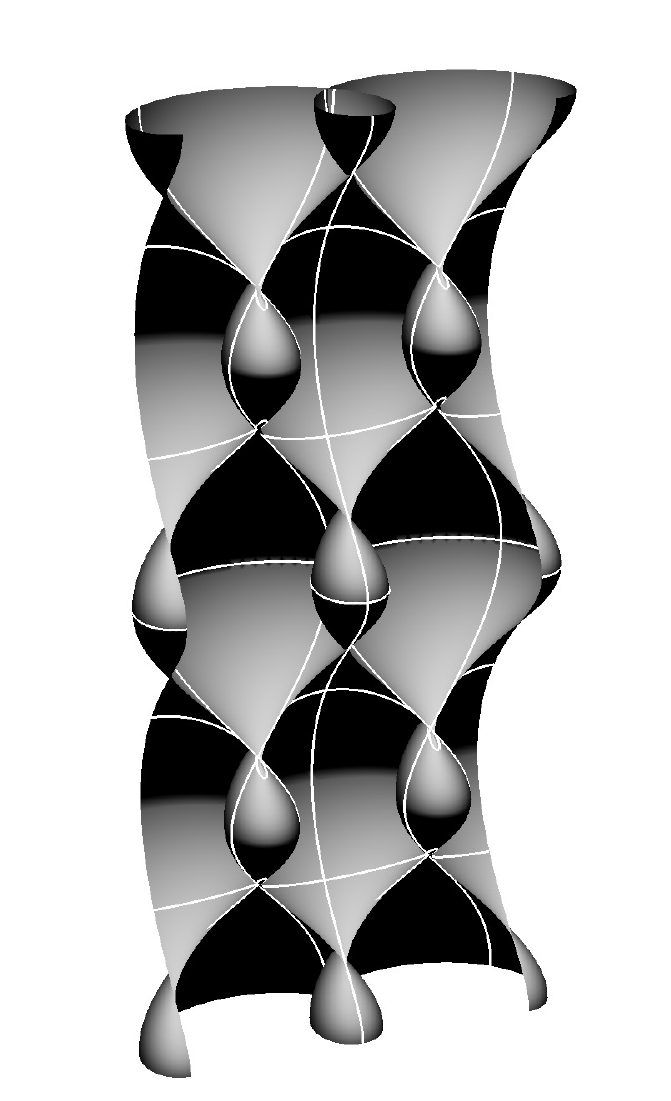}
\vspace{0.1cm}
\end{center}
\end{minipage}

\hspace{-3.5cm}
\begin{minipage}{0.5\hsize}
\begin{center}
\vspace{-1.1cm}
\includegraphics[clip,scale=0.45,bb=0 0 555 449]{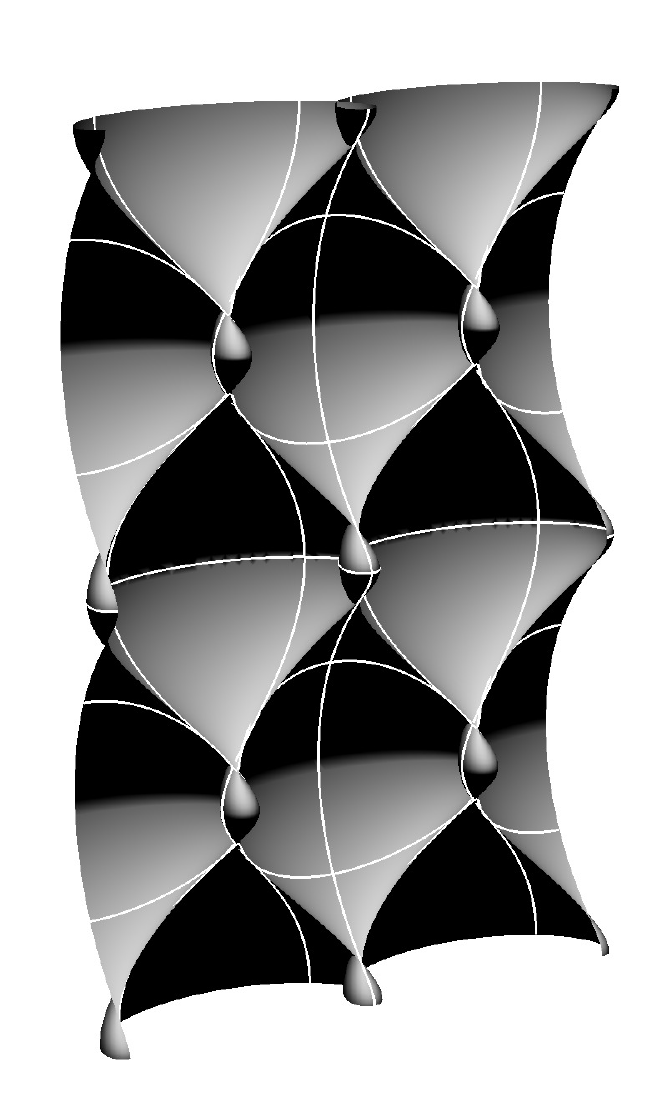}
\vspace{0.1cm}
\end{center}
\end{minipage}

\end{tabular}
\end{center}
\vspace{-0.8cm}
\caption{The singly periodic elliptic catenoid $f_1$ (left), the doubly periodic Bonnet type surfaces $f_{1.5}$ (center) and $f_{2}$ (right). In addition to translation symmetries, these surfaces share the planar symmetries and the line symmetry discussed in Example \ref{example:Bonnet}.}\label{Fig:Bonnet}

\end{figure}
\end{example}




\begin{thebibliography}{99}

\bibitem{A}
S.~Akamine, {\it Behavior of the Gaussian curvature of timelike minimal surfaces with singularities}, Hokkaido Math. J. {\bf 48} (2019), 537--568.

\bibitem{ACO}
 S.~Akamine, J.~Cho, Y.~Ogata, \emph{Analysis of timelike Thomsen surfaces}, J.\ Geom.\ Anal. {\bf 30}(1) (2020), 731--761. 
 
 \bibitem{AF}
S.~Akamine and H.~Fujino, \emph{Duality of boundary value problems for minimal and maximal surfaces}, To appear in Comm. Anal. Geom., arXiv:1909.00975.
  
 \bibitem{ACM1}
 L.J.~Al\'ias, R.M.B.~Chaves and P.~Mira, {\it Bj\"orling problem for maximal surfaces in Lorentz-Minkowski space}, Math. Proc. Cambridge Philos. Soc. {\bf 134} (2003), 289--316.
  
  \bibitem{AL}
 H.~Ara\'ujo and M.L.~Leite, {\it How many maximal surfaces do correspond to one minimal surface?}, Math. Proc. Cambridge Philos. Soc. {\bf 146} (2009), no. 1, 165--175.
 

\bibitem{Bonnet}
O. Bonnet, {\it M\'emoire sur la th\'eorie des surfaces applicables}, J. \'Ec Polyt. {\bf 42} (1867), 72--92.

\bibitem{CDM}
R.M.B.~Chaves, M.P.~Dussan, M.~Magid, \emph{Bj\"{o}rling problem for timelike surfaces in the Lorentz-Minkowski space}, J.\ Math.\ Anal.\ Appl.\ {\bf 377} (2011), no. 2, 481--494. 

\bibitem{Clelland}
  J.N.~Clelland, 
  \emph{Totally quasi-umbilic timelike surfaces in $\mathbb{R}^{1,2}$},
  Asian J.\ Math.\ {\bf 16}, 189--208 (2012).

\bibitem{DHS}
U.~Dierkes, S.~Hildebrandt, and F.~Sauvigny, {\it Minimal surfaces}, second, Grundlehren der Mathematischen Wissenschaften [Fundamental Principles of Mathematical Sciences], vol. {\bf 339}, Springer, Heidelberg, 2010.


\bibitem{FI}
A.~Fujioka and J.~Inoguchi,
\emph{Timelike Bonnet surfaces in Lorentzian space forms}, Differential Geom. Appl. {\bf 18}, 103--111 (2016).

\bibitem{G}
E.~Goursat, 
\emph{Sur un mode de transformation des surfaces minima}, Acta Math. {\bf 11} (1887),
no. 1--4, 257--264. Second M\'emoire. MR1554756

\bibitem{Graves}
L.~Graves, \emph{Codimension one isometric immersions between Lorentz spaces}, Trans. Amer. Math. Soc. {\bf 252} (1979), 367--392.


\bibitem{Lee}
H.~Lee, {\it Extensions of the duality between minimal surfaces and maximal surfaces}, Geom. Dedicata, 373--386 (2011). 

\bibitem{IT}
J.~Inoguchi and M.~Toda,
\emph{Timelike minimal surfaces via loop groups}, Acta Appl.\ Math.\ {\bf 63}, 313--355 (2004).

 \bibitem{Karcher}
 H.~Karcher, 
 \emph{Construction of minimal surfaces, Surveys in Geometry}, University of Tokyo, 1989,
and Lecture Notes No. 12, SFB 256, Bonn, 1989, pp. 1--96.

\bibitem{KKSY}
Y.~W.~Kim, S.-E.~Koh, H.~Shin and S.-D.~Yang,
\emph{Spacelike maximal surfaces, timelike minimal surfaces, and Bj\"orling representation formulae}, J.\ Korean Math.\ Soc.\ {\bf 48}, 1083--1100 (2011).

\bibitem{Konderak}
J.~Konderak,
\emph{A Weierstrass representation theorem for Lorentz surfaces}, Complex Var.\ Theory Appl.\ {\bf 50} (5), 319--332 (2005).


\bibitem{LM}
K.~Leschke and K.~Moriya, 
\emph{Simple factor dressing and the L\'opez-Ros deformation of minimal surfaces in Euclidean 3-space}, Math. Z. {\bf 291}, 1015--1058 (2019). 


\bibitem{LLS}
F.J.~L\'opez, R.~L\'opez, and R.~Souam, {\it Maximal surfaces of Riemann type in Lorentz-Minkowski space $\mathbb{L}^3$}, Michigan Math. J. {\bf 47} (2000), no. 3, 469--497. 

\bibitem{LR}
F.J.~L\'opez and A. Ros, 
\emph{On embedded complete minimal surfaces of genus zero}, 
J. Differential Geom. {\bf 33}(1), 293--300 (1991). 

\bibitem{McNertney}
L.~McNertney, 
\emph{One-parameter families of surfaces with constant curvature in Lorentz $3$-space}, Ph.D.\ thesis, Brown University (1980).

\bibitem{Meeks}
W.H.~Meeks, 
\emph{The Theory of Triply Periodic Minimal Surfaces}, 
Indiana University Mathematics Journal, {\bf 39}(3), 877--936 (1990).

  
\bibitem{Schwarz}  
H.A.~Schwarz,\emph{Gesammelte mathematische Abhandlungen} Vol.{\bf 1}. Springer, Berlin (1890).

\bibitem{Spivak}  
M.~Spivak, \emph{A comprehensive introduction to differential geometry}, Vol. {\bf IV}, Publish or perish Inc. (1970).

\bibitem{UY}
M.~Umehara and K.~Yamada,
\emph{Maximal surfaces with singularities in Minkowski space}, Hokkaido Math. J. {\bf 35}, 13--40 (2006).


\bibitem{W}
T.~Weinstein, 
\emph{An Introduction to Lorentz Surfaces}, de Gruyter Exposition in Math. {\bf 22}, Walter de Gruyter, Berlin (1996).

\end{thebibliography}
 \end{document}